\documentclass[11pt]{amsart}
\usepackage{amssymb}
\usepackage{hyperref}

\newcommand{\R}{{\mathbb R}}

\numberwithin{equation}{section}
\newtheorem{theorem}{Theorem}[section]
\newtheorem{lemma}[theorem]{Lemma}
\newtheorem{rmk}[theorem]{Remark}
\newtheorem{cor}[theorem]{Corollary}

\newtheorem{example}[theorem]{Example}
\allowdisplaybreaks
\def\eqref#1{(\ref{#1})}

\def\color#1{}

\begin{document}
\date{February 2018}
\author{Alexander Grigor'yan}
\address{Department of Mathematics, University of Bielefeld, 33501
Bielefeld, Germany}
\email{grigor@math.uni-bielefeld.de}
\author{Igor Verbitsky}
\address{Department of Mathematics, University of Missouri, Columbia,
Missouri 65211, USA}
\email{verbitskyi@missouri.edu}
\title[Pointwise estimates of solutions for nonlocal operators]
{Pointwise estimates of solutions to nonlinear equations for nonlocal
operators}

\thanks{%
The work is supported by German Research Council through SFB 1283.}


\begin{abstract}
We study pointwise behavior of positive solutions to nonlinear integral
equations, and related inequalities, of the type 
\begin{equation*}
u(x) - \int_\Omega G(x, y) \, g(u(y)) d \sigma (y) = h,
\end{equation*}
where $(\Omega, \sigma)$ is a locally compact measure space, $G(x, y)\colon
\Omega\times \Omega \to [0, +\infty]$ is a kernel, $h \ge 0$ is a measurable
function, and $g\colon [0, \infty)\to [0, \infty)$ is a monotone {\color{red} increasing} function.

This problem is motivated by the semilinear fractional Laplace equation 
\begin{equation*}
(-\Delta)^{\frac{\alpha}{2}} u - g(u) \sigma = \mu \quad \text{in} \, \,
\Omega, \quad u=0 \, \, \, \text{in} \, \, \Omega^c,
\end{equation*}
with measure coefficients $\sigma$, $\mu$, where $g(u)=u^q$, {\color{red} $q>0$}, and $0<\alpha<n$, in domains $\Omega \subseteq {\mathbb{%
R}}^n$, or Riemannian manifolds, with positive Green's function $G$. 
{\color{red} In a similar way, we treat positive solutions to the equation 
\begin{equation*}
u(x) + \int_\Omega G(x, y) \, g(u(y)) d \sigma (y) = h,
\end{equation*}
and the corresponding fractional Laplace equation $(-\Delta)^{\frac{\alpha}{2}} u + g(u) \sigma = \mu$, with a  monotone decreasing function $g$, in particular $g(u)=u^q$, $q<0$.}
\end{abstract}

\maketitle
\tableofcontents

\section{Introduction}

\label{intro}

We study pointwise behavior of non-negative solutions $u$ to nonlinear
integral equations (and related inequalities) of the type 
\begin{equation}
u(x)=\int_{\Omega }G(x,y)\,g(u(y))d\sigma (y)+h\left( x\right), \quad x \in
\Omega .  \label{integral}
\end{equation}%
Here $\Omega $ is a locally compact Hausdorff space, $\sigma $ is a Radon
measure on $\Omega $, $G(x,y)\colon \Omega \times \Omega \rightarrow
\lbrack 0,+\infty \rbrack$ is a lower semicontinuous function, $g\colon \lbrack 0,\infty
)\rightarrow \lbrack 0,\infty )$ is a monotone {\color{red} increasing} continuous function, and $%
h\geq 0$ is a given $\sigma $-measurable function.

We also treat positive solutions to more general equations and inequalities
with measure-valued kernels $G$, or equivalently, operator equations of the
type 
\begin{equation}
u= T(g(u)) +h \quad \text{in} \, \, \Omega  \label{operator}
\end{equation}
on a measurable space $(\Omega, \Xi )$, where $T$ is a positivity preserving
linear operator acting in the cone of measurable functions $\Omega
\rightarrow \lbrack 0,+\infty \rbrack$, and continuous with respect to
pointwise limits of monotone increasing sequences of functions (see Sec.~\ref%
{weakmaximum}, Examples \ref{Ex1}--\ref{Ex3}).

Our main applications are to positive weak solutions of the semilinear
fractional Laplace problem with measure coefficients, 
\begin{equation}
(-\Delta )^{\frac{\alpha }{2}}u-u^{q}\sigma =\mu \quad \text{in}\,\,\Omega,
\quad u=0\,\,\text{in}\,\,\Omega ^{c},  \label{differential}
\end{equation}%
where {\color{red} $q>0$}, $0<\alpha <n$, and $\mu $, $\sigma 
$ are Radon measures in a domain $\Omega \subseteq {\R}^{n}$ (or a
Riemannian manifold) with positive Green's function $G$. Equation (\ref%
{differential}) is equivalent to the integral equation (\ref{integral}) with 
\begin{equation*}
h=G\mu =\int_{\Omega }G\left( \cdot,y\right) d\mu \left( y\right) .
\end{equation*}

{\color{red} Similar results are obtained for the integral equation 
\begin{equation}
u(x) = -\int_{\Omega }G(x,y)\,g(u(y))d\sigma (y)+  h\left( x\right), \quad x \in
\Omega,   \label{integral-2}
\end{equation}
where $g$ is a monotone decreasing function, and the corresponding fractional Laplace problem  
\begin{equation}
(-\Delta )^{\frac{\alpha }{2}}u+ u^{q}\sigma =\mu \quad \text{in}\,\,\Omega,
\quad u=0\,\,\text{in}\,\,\Omega ^{c},  \label{differential-2}
\end{equation}%
in the case $q<0$.}

More precisely, we obtain sharp global lower bounds for non-negative
measurable functions $u$ satisfying the integral inequality 
\begin{equation}
u(x)\geq G(g\left( u\right) d\sigma )(x)+h(x)\quad \text{in}\,\,\Omega   \label{q>0}
\end{equation}%
in the case of monotone increasing $g$, and upper bounds for solutions of 
\begin{equation}
u(x)\leq {\color{red} -G(g\left( u\right) d\sigma )(x)}+h(x)\quad \text{in}\,\,\Omega 
\label{q<0}
\end{equation}%
in the case of monotone decreasing $g$.

We assume that the kernel $G$ satisfies the following form of the \emph{weak
maximum principle}:

\emph{For any Radon measure $\nu $ in $\Omega $ with compact support $K=%
\mathrm{supp} (\nu)$,%
\begin{equation}
G\nu \leq 1\ \text{in }K\ \Longrightarrow \ G\nu \leq \mathfrak{b}\ \text{in 
} \Omega ,  \label{wmp}
\end{equation}
with some constant $\mathfrak{b}\geq 1$.}

This property of $G$ is sometimes referred to as the generalized, or rough
maximum principle (see \cite{AH}, \cite{K1}, \cite{L}), and is known for
many local and non-local operators.

In the case $h=1$ our main result is the following theorem.

\begin{theorem}
\label{mainthm:h=1} Let $G$ satisfy in $\Omega $ the weak maximum principle 
\emph{(\ref{wmp})} with constant $\mathfrak{b}\geq 1$. Assume that $g$ is a
monotone non-decreasing positive continuous function in $\lbrack 1,
+\infty\rbrack$, and set 
\begin{equation*}
F(t)=\int_{1}^{t}\frac{ds}{g(s)},\quad t\geq 0.
\end{equation*}%
If $u$ is a positive measurable function on $\Omega $ that satisfies 
\begin{equation}
u\left( x\right) \geq G(g\left( u\right) d\sigma )\left( x\right) +1
\label{uG}
\end{equation}%
for $\sigma $-almost all $x\in \Omega $, then, at any point $x\in \Omega $
where \emph{(\ref{uG})} is satisfied and $u\left( x\right) <+\infty $, we
have 
\begin{equation*}
G\sigma (x)<\mathfrak{b}\,F(\infty )=\mathfrak{b}\int_{1}^{+\infty }\frac{ds%
}{g(s)},
\end{equation*}%
and%
\begin{equation*}
u(x)\geq 1+\mathfrak{b}\left[ F^{-1}\Big(\mathfrak{b}^{-1}G\sigma (x)\Big)-1%
\right] .
\end{equation*}
\end{theorem}

For a similar result in a more general setup of measure-valued kernels
(Example \ref{Ex1}),  see Theorem \ref{theorem2a}; non-increasing
nonlinearities $g$ are treated in  Theorem \ref{theorem3a}.

For example, for $g\left( s\right) =s^{q}$ with $q>0$, $q\neq 1$, we obtain 
\begin{equation*}
u(x)\geq 1+\mathfrak{b}\Big[\Big(1+\frac{(1-q)\,G\sigma (x)}{\mathfrak{b}}%
\Big)^{\frac{1}{1-q}}-1\Big],
\end{equation*}%
where in the case $q>1$ necessarily 
\begin{equation*}
G\sigma (x)<\frac{\mathfrak{b}}{q-1}.
\end{equation*}%
In the case $q=1$ we have 
\begin{equation*}
u(x)\geq 1+\mathfrak{b}\,\Big(e^{\mathfrak{b}^{-1}\,G\sigma \left( x\right)
}-1\Big).
\end{equation*}

Theorem~\ref{mainthm:h=1} and its corollaries yield global pointwise
estimates for positive solutions to a broad class of elliptic and parabolic
PDE, as well as non-local problems on Euclidean domains and manifolds. In
particular, they are applicable to convolution equations on ${\R}^n$ of
the type 
\begin{equation}
u = k\star (g(u) d \sigma) +1.  \label{conv}
\end{equation}
Here 
\begin{equation*}
k\star (f d \sigma)(x) =\int_{{\R}^n} k(x-y) f(y) d \sigma(y),
\end{equation*}
where $k=k(|x|)>0$ is an arbitrary lower semi-continuous, radially non-increasing function on $\R^n$. Such kernels are
known to satisfy the weak maximum principle (\ref{wmp}) with constant 
$\mathfrak{b}$ which depends only on the dimension $n$ of the underlying
space (see \cite{AH}, Theorem 2.6.2).

To treat more general right-hand sides $h$, we invoke a weak form of the 
\textit{domination principle} for kernels $G$ with respect to a Radon
measure $\sigma$ on $\Omega$:

\emph{For any bounded measurable function $f$ with compact support,} 
\begin{equation}
G (f d \sigma)(x)\leq h(x)\ \text{in }\mathrm{supp} (f)\ \ \Longrightarrow \ \
G(f d \sigma)(x) \leq \mathfrak{b}\ h(x) \, \, \text{in\ }\Omega.
\label{dom-intr}
\end{equation}

Closely related properties are sometimes referred to as the dilated
domination principle, the complete maximum principle, or the 
second maximum principle in the case $\mathfrak{b}=1$ (see \cite{BH}, \cite{K2}, \cite{L}).

We observe that the weak domination principle holds for Green's kernels
associated with a large class of elliptic and parabolic problems, along with
many non-local operators, in particular, integral operators with the
so-called \textit{quasi-metric}, or quasi-metrically modifiable kernels,
treated in Sec. \ref{quasimetric} (see also \cite{FNV}, \cite{GH}, \cite{H}, \cite{HN}, \cite{KV}).

Our main estimates in the case of general $h$ and $g(t)=t^q$, $q \in \R\setminus\{0\}$, are contained in the following theorem. A more
general statement in the context of measure-valued kernels $G$ can be found
in Theorem \ref{thm3.1} below.

\begin{theorem}
\label{thm-main} Let $h>0$ be a lower semi-continuous function in $\Omega$.
Let $G$ be a non-negative kernel in $\Omega\times \Omega$ which satisfies
the domination principle \emph{(\ref{dom-intr})} with respect to $h$ with
constant $\mathfrak{b}\ge 1$. Suppose that $u$ is a non-negative function
such that $u^q \in L^1_{\mathrm{loc}} (\sigma)$, which satisfies \emph{(\ref%
{q>0})} if $q>0$, and \emph{(\ref{q<0})} if $q<0$.

Then if $q>0$ $(q \not=1)$, we have 
\begin{equation}  \label{sharp-1}
u(x) \ge h(x) \left\{ 1+ \mathfrak{b} \Big[ \Big(1 + \frac{(1-q) \, G (h^q
d\sigma)(x)}{ \mathfrak{b} \, h(x)}\Big)^{\frac{1}{1-q}} -1 \Big] \right\},
\quad x \in \Omega,
\end{equation}
where in the case $q>1$ necessarily 
\begin{equation}  \label{necessary-1}
G (h^q d\sigma)(x) < \frac{ \mathfrak{b}}{q-1} \, h(x),
\end{equation}
for all $x \in \Omega$ such that $0\le u(x)<+\infty$, and \emph{(\ref{q>0})} holds.

In the case $q=1$, 
\begin{equation}  \label{sharp-2}
u(x) \ge h(x) \Big[ 1+ \mathfrak{b} \, \Big(e^{\mathfrak{b}^{-1} \, \frac{G
(h d \sigma)(x)}{h(x)}} -1\Big)\Big], \quad x \in \Omega.
\end{equation}

If $q<0$, then 
\begin{equation}  \label{sharp-3}
u(x) \le h(x) \left\{ 1- \mathfrak{b} \Big[ 1-\Big(1 - \frac{(1-q) \, G (h^q
d\sigma)(x)}{\mathfrak{b} \, h(x)}\Big)^{\frac{1}{1-q}}\Big] \right\}, \quad
x \in \Omega,
\end{equation}
and necessarily 
\begin{equation}  \label{necessary-3}
G (h^q d\sigma)(x) < \frac{\mathfrak{b}}{1-q} \Big[1- (1-\mathfrak{b}%
^{-1})^{1-q}\Big] h(x),
\end{equation}
for all $x\in \Omega$ such that $u(x)>0$, $h(x)<+\infty$, and \emph{(\ref{q<0})}
holds.
\end{theorem}

These estimates {\color{red} with $h = G \mu$} yield the corresponding lower bounds for positive weak
solutions $u \in L^q_{\mathrm{loc}} (\sigma)$ to the inequality 
\begin{equation*}
(-\Delta)^{\frac{\alpha}{2}} u - u^q \sigma \ge \mu \quad \text{in} \, \,
\Omega, \quad u= 0 \, \, \text{in} \, \, \Omega^c,
\end{equation*}
for $q>0$, and the upper bounds for positive weak solutions to the
inequality 
\begin{equation*}
(-\Delta)^{\frac{\alpha}{2}} u {\color{red}+ u^q \sigma} \le \mu \quad \text{in} \, \,
\Omega, \quad u= 0 \, \, \text{in} \, \, \Omega^c,
\end{equation*}
for $q<0$, where $\mu, \sigma$ are Radon measures in an arbitrary domain $%
\Omega\subseteq {\R}^n$ with positive Green's function $G$ in the
case $0<\alpha \le 2$, when the domination principle holds with $\mathfrak{b}%
=1$. They also hold in the case $0<\alpha <n$ (with some constant $\mathfrak{%
b}\ge 1$) provided Green's function $G$ is quasi-metric, or quasi-metrically
modifiable, for instance, if $\Omega$ is the entire space, a ball, or
half-space (see \cite{FNV}).

For classical solutions and local elliptic differential operators, such
estimates were obtained earlier in \cite{GV} in the case where $\sigma \in
C(\Omega)$ is a continuous function which may change sign, and $\mu \ge 0$
is a locally H\"older continuous function in $\Omega\subseteq \R^n$, or 
a smooth Riemannian manifold.

In the linear case $q=1$, estimate (\ref{sharp-2}) in Theorem~\ref{thm-main}
obviously yields 
\begin{equation*}
u(x) \ge h(x) \, e^{\mathfrak{b}^{-1} \, \frac{G (h d \sigma)(x)}{h(x)}},
\quad x \in \Omega.
\end{equation*}
This is a refinement (with sharp constant in the case $\mathfrak{b}=1$) of
the lower bound obtained in \cite{FNV} along with a matching upper bound,
for quasi-metric kernels $G$. See Sections \ref{quasimetric} and \ref%
{supersol} where these and more general classes of kernels are treated.

In the superlinear case $q>1$, the necessary condition (\ref{necessary-1})
was found in \cite{KV} for quasi-metric kernels $G$ (without the explicit
constant), and in \cite{BC} for the Laplace operator $-\Delta$ with
Dirichlet boundary conditions, with sharp constant $\frac{1}{q-1}$. In the
latter case, the domination principle holds with $\mathfrak{b}=1$, so that
this constant is the same as in (\ref{necessary-1}).

There are certain upper bounds for $u$ in the case $q>1$, and lower bounds
in the case $q<0$, which are true for general non-negative kernels $G$
(without the weak domination principle) provided conditions (\ref%
{necessary-1}) and (\ref{necessary-3}) respectively hold with 
smaller constants depending on $q$, which ensure the existence of positive solutions (\cite{GV}, Theorem 3.5; see also \cite{BC}, \cite{GS}, \cite{KV} for $q>1$).

For the \textit{homogeneous} problem (\ref{q>0}) with $h=0$ in the sublinear
case $0<q<1$, we have a similar lower estimate for non-trivial  solutions $u$
(see Corollary~\ref{thm2.1} for a more general setup of measure-valued
kernels). No such estimates of positive solutions to (\ref{q>0}) or (\ref%
{q<0}) with $h=0$ are available for $q\geq 1 $ or $q<0$, respectively.

\begin{theorem}
\label{thm-main2} Let $0<q<1$. Suppose $G$ is a non-negative kernel on $%
\Omega \times \Omega$ satisfying the weak maximum principle \emph{(\ref{wmp})%
} with constant $\mathfrak{b}\ge 1$. If $u \in L^q_{\mathrm{loc}} (\sigma)$, 
$u>0$, is a  solution to the integral inequality \emph{(\ref{q>0})} with $h=0
$, that is, $u \ge G(u^q d \sigma)$ in $\Omega$, then 
\begin{equation}  \label{lower-2}
u (x) \ge (1-q)^{\frac{1}{1-q}} \mathfrak{b}^{-\frac{q}{1-q}} \Big[ G \sigma
(x)\Big]^{\frac{1}{1-q}}, \quad x \in \Omega.
\end{equation}
\end{theorem}

If $\mathfrak{b}=1$, then the constant $(1-q)^{\frac{1}{1-q}}$ in (\ref%
{lower-2}) is the same as in the local case (\cite{GV}, Theorem 3.3), and is
sharp. This estimate was first deduced in \cite{BK} for solutions to the
equation $-\Delta u= u^q \, \sigma$ on ${\R}^n$ (without the sharp
constant). See also \cite{CV1}, \cite{CV2} for matching lower and upper
bounds of solutions to the equation $(-\Delta)^{\frac{\alpha}{2}} u= u^q \,
\sigma$ ($0<\alpha<n$) on ${\R}^n$ in the case $0<q<1$.

We remark that necessary and sufficient conditions for the existence of a
positive solution in the
case $0<q<1$ to the homogeneous equation $u = G(u^q d \sigma)$ in $\Omega$ for quasi-symmetric kernels $G$ which satisfy the weak maximum
principle are given in \cite{QV2} (see also \cite{QV1}). 

{\bf Acknowledgement.} The authors are grateful to Alexander Bendikov and Wolfhard Hansen for stimulating discussions.
The second author wishes to thank the Mathematics Department at Bielefeld 
University for the hospitality during his visits.

\section{The weak maximum principle and iterated estimates}

\label{weakmaximum}

We first prove a series of lemmas.

\begin{lemma}
\label{lemma1.2}{Let }$\left( \Omega ,\omega \right) ${\ be a $\sigma $%
-finite measure space, and let $a=\omega (\Omega )\leq +\infty $.} Let $%
f\colon \Omega \rightarrow {[0,+\infty ]}$ be a measurable function. Let $%
\phi \colon \lbrack 0,a)\rightarrow \lbrack 0,+\infty )$ be a continuous,
monotone non-decreasing function, {and set $\phi (a)\mathrel{\mathop:}%
=\lim_{t\rightarrow a^{-}}\phi (t)\in (0,+\infty ]$}. Then the following
inequality holds: 
\begin{equation}
\int_{0}^{\omega \left( \Omega \right) }\phi (t)\,dt\leq \int_{\Omega }\phi
\left( \omega \left( \{z\in \Omega \colon \,f(z)\leq f(y)\}\right) \right)
\,d\omega (y).  \label{byparts}
\end{equation}
\end{lemma}

\begin{proof}
If $\omega \left( \Omega \right) =\infty $ then consider an exhausting
sequence $\left\{ \Omega _{k}\right\} _{k=1}^{\infty }$ of measurable
subsets of $\Omega $ with $\omega \left( \Omega _{k}\right) <+\infty $.
Since 
\begin{align*}
& \int_{\Omega }\phi \left( \omega \left( \{z\in \Omega \colon \,f(z)\leq
f(y)\}\right) \right) d\omega \left( y\right) \\ & \geq \int_{\Omega _{k}}\phi
\left( \omega \left( \{z\in \Omega _{k}\colon \,f(z)\leq f(y)\}\right)
\right) d\omega \left( y\right) ,
\end{align*}%
it suffices to prove (\ref{byparts}) for $\Omega _{k}$ instead of $\Omega $
and then let $k\rightarrow \infty $. {Hence, we can assume without loss of
generality that } $a=\omega \left( \Omega \right) <\infty $.

If $\phi $ is unbounded (as $t\rightarrow a^{-}$) then replace $\phi $ by a
bounded function {$\phi _{j}=\min (\phi ,j)$, prove }(\ref{byparts}) for $%
\phi _{j}$ and then let $j\rightarrow \infty $. Therefore, we assume in what
follows that $\phi $ is continuous on $[0,a]$.

Assume first that the function $f$ is simple, that is, $f$ takes on only a
finite set of values $\left\{ b_{i}\right\} _{i=1}^{N}$ where the sequence $%
\left\{ b_{i}\right\} $ is arranged in the increasing order. Set%
\begin{equation*}
a_{i}=\omega \left\{ z\in \Omega :f\left( z\right) \leq b_{i}\right\} \ 
\text{for }i=1,...,N
\end{equation*}%
and $a_{0}=0$. If $y\in \Omega $ is such that $f\left( y\right) =b_{i}$, then%
\begin{equation*}
\phi \left( \omega \{z\in \Omega \colon \,f(z)\leq f(y)\}\right) \,=\phi
\left( a_{i}\right) .
\end{equation*}%
It follows that%
\begin{eqnarray*}
&&\int_{\Omega }\phi \left( \omega \{z\in \Omega \colon \,f(z)\leq
f(y)\}\right) \,d\omega (y) \\
&=&\sum_{i=1}^{N}\int_{\{y\in \Omega :f\left( y\right) =b_{i}\}}\phi \left(
\omega \{z\in \Omega \colon \,f(z)\leq f(y)\}\right) \,d\omega (y) \\
&=&\sum_{i=1}^{N}\phi \left( a_{i}\right) \omega \{y\in \Omega :f\left(
y\right) =b_{i}\} \\
&=&\sum_{i=1}^{N}\phi \left( a_{i}\right) (a_{i}-a_{i-1}).
\end{eqnarray*}%
Since the function $\phi $ is non-decreasing, we obtain%
\begin{equation*}
\sum_{i=1}^{N}\phi \left( a_{i}\right) (a_{i}-a_{i-1})\geq
\sum_{i=1}^{N}\int_{a_{i-1}}^{a_{i}}\phi \left( t\right)
dt=\int_{a_{0}}^{a_{N}}\phi \left( t\right) dt=\int_{0}^{\omega \left(
\Omega \right) }\phi \left( t\right) dt,
\end{equation*}%
whence (\ref{byparts}) follows.

Let $f$ be an arbitrary measurable function. By a standard argument, there
is an increasing sequence $\left\{ f_{n}\right\} $ of simple functions such
that $f_{n}\uparrow f$  as $n\rightarrow \infty $. Applying the first part
of the proof to $f_{n}$, we obtain%
\begin{align*}
& \int_{\Omega }\phi \left( \omega \{z\in \Omega \colon \,f_{n}(z)\leq
f(y)\}\right) \,d\omega \left( y\right)  \\ & \geq \int_{\Omega }\phi \left(
\omega \{z\in \Omega \colon \,f_{n}(z)\leq f_{n}(y)\}\right) \,d\omega
\left( y\right)  \\
&\geq \int_{0}^{\omega \left( \Omega \right) }\phi \left( t\right) dt.
\end{align*}%
The sequence $F_{n}(y)\mathrel{\mathop:}=\omega \{z\in \Omega \colon
\,f_{n}(z)\leq f(y)\}\,$ of functions of $y\in \Omega $ is decreasing in $n$
and converges to $\omega \{z\in \Omega \colon \,f(z)\leq f(y)\}\,$ as $%
n\rightarrow \infty $. Since $\phi $ is bounded and continuous and $\omega
\left( \Omega \right) <\infty $, we obtain by the bounded convergence
theorem that 
\begin{align*}
& \lim_{n \to \infty} \int_{\Omega }\phi \left( \omega \{z\in \Omega \colon \,f_{n}(z)\leq
f(y)\}\right) \,d\omega \left( y\right) \\ & = \int_{\Omega }\phi
\left( \omega \{z\in \Omega \colon \,f(z)\leq f(y)\}\right) \,d\omega \left(
y\right) ,
\end{align*}%
whence (\ref{byparts}) follows.
\end{proof}

In the rest of the section we assume that $(\Omega ,\Xi )$ is a measurable
space, and that $G\left( x,dy\right) $ is a \emph{$\sigma $-finite kernel} in $\Omega $,
which means that, for any $x\in \Omega $, $G\left( x,dy\right) $ is a $%
\sigma $-finite measure on $(\Omega ,\Xi )$, and this measure depends on $x$
measurably. The latter means that, for any $\Xi $-measurable function $%
f:\Omega \rightarrow \lbrack 0,\infty ]$, the function%
\begin{equation*}
Gf\left( x\right) =\int_{\Omega }f\left( y\right) G\left( x,dy\right)
\end{equation*}%
is also $\Xi $-measurable. We assume here that $G$ satisfies the \textit{weak
maximum principle} in the following form: 

\emph{There is a constant $\mathfrak{b}\geq 1$ such that, for
any non-negative bounded measurable function $f$ on $\Omega $,} 
\begin{equation}
Gf\leq 1\ \ \text{in }\left\{ f>0\right\} \ \ \Longrightarrow \text{ }Gf\leq 
\mathfrak{b}\ \text{in }\Omega .\text{ }  \label{wmp-f}
\end{equation}%

Clearly, (\ref{wmp-f}) implies the following: for any $\varepsilon \geq 0$,%
\begin{equation}
Gf\leq \varepsilon \ \ \text{in }\left\{ f>0\right\} \ \Longrightarrow \ Gf\leq
\varepsilon \mathfrak{b}\ \text{in }\Omega .  \label{wmp-f-ep}
\end{equation}%
Let us consider some examples where this framework is applicable.

\begin{example}
\label{Ex1}{\rm Let $\Omega $ be a locally compact Hausdorff space with
countable base. Assume that $G\left( x,dy\right) $ is a \emph{Radon kernel} in $%
\Omega $, that is, for any $x\in \Omega $, $G\left( x,dy\right) $ is a Radon
measure on $\Omega $ (in particular, $\sigma $-finite). The weak maximum
principle for a Radon kernel $G$ can be stated as follows:

\emph{For any bounded
measurable function $f$ with compact support,}
\begin{equation}
Gf\leq 1\ \text{in }\mathrm{supp} (f)\ \ \Longrightarrow \ \ Gf\leq \mathfrak{b}\ 
\text{in\ }\Omega .  \label{wmp-c}
\end{equation}%
Then the weak maximum principle holds also in the form (\ref{wmp-f}) by
approximating an arbitrary function $f$ by $f1_{F}$ with compact $F$.}
\end{example}

\begin{example}
\label{Ex2} {\rm Let $\Omega $ again be a locally compact Hausdorff space with
countable base, and let $\omega $ be a Radon measure on $\Omega $. Suppose $%
K:\Omega \times \Omega \rightarrow \lbrack 0,+\infty ]$ is a measurable
function such that $K\left( x,\cdot \right) \in L_{loc}^{1}\left( \Omega
,\omega \right) $ for any $x\in \Omega $. Then set 
\begin{equation}
G(x,E):=\int_{E}K(x,y)d\omega (y),  \label{ex-ker}
\end{equation}%
so that $G$ is a Radon kernel in the above sense. In this case we have} 
\begin{equation*}
Gf(x)=\int_{\Omega }K(x,y)\,f(y)d\omega (y).
\end{equation*}
\end{example}

\begin{example}
\label{Ex3} {\rm  Let $(\Omega ,\Xi )$ be a measurable space. Let $T$ be a
positivity preserving linear operator acting in the cone of measurable
functions $\Omega \rightarrow \left[ 0,+\infty \right] $. Assume also that $%
T $ is continuous with respect to pointwise limits of monotone increasing
sequences of functions.

Fix some $x\in \Omega $ and define the measure 
\begin{equation*}
\omega _{x}\left( E\right) =T1_{E}\left( x\right) ,\quad E\in \Xi .
\end{equation*}%
Then, for any non-negative measurable function $f$ on $\Omega $, we have%
\begin{equation}
Tf\left( x\right) =\int_{\Omega }fd\omega _{x}.  \label{Gu}
\end{equation}%
Indeed, if $u$ is a simple function of the form 
\begin{equation*}
f=\sum_{i=1}^{\infty }f_{i}1_{E_{i}},
\end{equation*}%
where $f_{i}$ are non-negative constants, then%
\begin{equation*}
Tf\left( x\right) =\sum_{i=1}^{\infty }f_{i}T1_{E_{i}}\left( x\right)
=\sum_{i=1}^{\infty }u_{i}\omega _{x}\left( E_{i}\right) =\int_{\Omega
}fd\omega _{x}.
\end{equation*}%
For a general measurable function $f$ one proves (\ref{Gu}) by passing to
the limit using an increasing sequence of simple functions.}
\end{example}

The measure $\omega _{x}$ is clearly $\sigma $-additive. The $\sigma $%
-finiteness of $\omega _{x}$ has to be assumed in addition (for example, if $%
T1\left( x\right) <\infty $ then the measure $\omega _{x}$ is finite).
Assuming that $\omega _{x}$ is $\sigma $-finite, set $G(x,dy):=d\omega
_{x}(y)$, or equivalently, 
\begin{equation}
G\left( x,E\right) =\int_{E}G(x,dy)=T1_{E}(x),\ \ \,E\in \Xi .
\label{ex-ker-2}
\end{equation}

The following is the key lemma used repeatedly throughout this paper.

\begin{lemma}
\label{lemma1.3a}Let $(\Omega ,\Xi )$ be a measurable space, and let $G\left(
x,dy\right) $ be a $\sigma $-finite kernel in $\Omega $. Assume that $G$
satisfies the weak maximum principle \emph{(\ref{wmp-f})}. Fix $x\in \Omega $
and set $a=G1(x)\leq +\infty $. If a function $\phi \colon \lbrack
0,a)\rightarrow \lbrack 0,+\infty )$ satisfies the conditions of Lemma~\emph{%
\ref{lemma1.2}} and $\phi (t)=\phi (a)$ for $t\geq a$, then 
\begin{equation}
\int_{0}^{G1(x)}\phi (t)dt\leq G\left[ \phi (\mathfrak{b}\,G1)\right]
(x)\quad \text{for\thinspace \thinspace all}\,\,\, x\in \Omega .  \label{iter3b}
\end{equation}
\end{lemma}

\begin{proof}
For any $y\in \Omega $, set 
\begin{equation*}
E_{y}=\{z\in \Omega \colon \,\,G1(z)\leq \,G1(y)\}.
\end{equation*}%
Clearly, 
\begin{equation*}
G1_{E_{y}}(z)\leq G1(z)\leq G1(y)\quad \text{for\thinspace \thinspace
all\thinspace }z\in E_{y}.
\end{equation*}%
Hence, by the weak maximum principle (\ref{wmp-f-ep}) applied to $f=1_{E_{y}}$,
we obtain 
\begin{equation*}
G1_{E_{y}}\left( z\right) \leq \mathfrak{b}\,G1(y)\quad \text{for\thinspace
\thinspace all\thinspace }\,z\in \Omega ,
\end{equation*}%
which implies, for $z=x$, 
\begin{equation}
G1_{E_{y}}(x)\leq \mathfrak{b\,}G1(y).  \label{G1E}
\end{equation}

Applying Lemma \ref{lemma1.2} to the $\sigma $-finite measure $\omega \left(
dy\right) =G\left( x,dy\right) $ and the function $f=G1$, noticing that $%
\omega (E)=G1_{E}(x)$, and using (\ref{G1E}), we obtain%
\begin{eqnarray*}
\int_{0}^{G1(x)}\phi (t)\,dt &\leq &\int_{\Omega }\phi \left( \omega \{z\in
\Omega \colon \,G1(z)\leq G1(y)\}\right) \,\omega (dy) \\
&=&\int_{\Omega }\phi \left( \omega \left( E_{y}\right) \right) \omega
\left( dy\right) \\
&\leq &\int_{\Omega }\phi \left( \mathfrak{b}\,G1(y)\right) \omega \left(
dy\right) \\
&=&G\left[ \phi \left( \mathfrak{b}\,G1\right) \right] \left( x\right) ,
\end{eqnarray*}%
which proves (\ref{iter3b}).
\end{proof}

\begin{rmk}
{\rm In the particular case $\phi (t)=t^{r-1}$ $(t\geq 0)$ where $r\geq 1$,
Lemma \ref{lemma1.3a} gives the following estimate: 
\begin{equation}
\left[ G1(x)\right] ^{r}\leq r\,\mathfrak{b}^{r-1}G\left[ (G1)^{r-1}\right]
(x)\quad \mathrm{for}\,\,\mathrm{all}\,\,x\in \Omega .  \label{iter3}
\end{equation}%
In the case $\phi (t)=t^{r-1}$ $(t>0)$ where $0<r\leq 1$, it is easy to see
that the converse inequality to (\ref{iter3}) holds, that is,} 
\begin{equation}
\left[ G1(x)\right] ^{r}\geq r\,\mathfrak{b}^{r-1}G\left[ (G1)^{r-1}\right]
(x)\quad \mathrm{for}\,\,\mathrm{all}\,\,x\in \Omega .  \label{iter3-conv}
\end{equation}
\end{rmk}

\begin{lemma}
\label{lemma1.4} In the setting of Lemma \emph{\ref{lemma1.3a}} define a
sequence $\left\{ f_{k}\right\} _{k=0}^{\infty }$ of functions on $\Omega $
by%
\begin{equation}
f_{0}=G1,\ \ \ f_{k+1}=G\left( \phi \left( f_{k}\right) \right) ,  \label{fT}
\end{equation}%
for all $k\geq 0$. Set 
\begin{equation}
\psi (t)=\phi (\mathfrak{b}^{-1}t),
  \label{def-psi}
\end{equation}%
and define also the sequence $\left\{ \psi _{k}\right\} _{k=0}^{\infty }$ of
functions on $[0,\infty )$ by\ $\psi _{0}(t)=t$ and 
\begin{equation}
\psi _{k+1}(t)=\int_{0}^{t}\psi \circ \psi _{k}(s)ds,  \label{psi-j}
\end{equation}%
for all $k\geq 0.$ Then, for all $x\in \Omega $ and $k\geq 0$, 
\begin{equation}
\psi _{k}\left( f_{0}\left( x\right) \right) \leq f_{k}(x).  
\label{iter-psi}
\end{equation}
\end{lemma}

For example, we have%
\begin{equation*}
f_{1}=G\left( \phi \left( f_{0}\right) \right) ,\ \ \ f_{2}=G\left( \phi
\left( f_{1}\right) \right) ,
\end{equation*}%
and%
\begin{equation*}
\psi _{1}\left( t\right) =\int_{0}^{t}\psi \left( s\right) ds,\ \ \ \ \psi
_{2}\left( t\right) =\int_{0}^{t}\psi \left( \psi _{1}\left( s\right)
\right) ds.
\end{equation*}

\begin{proof}
For $k=0$ estimate (\ref{iter-psi}) is trivial. For $k=1$ estimate (\ref%
{iter-psi}) follows from Lemma~\ref{lemma1.3a} with $\psi $ in place of $%
\phi $, since by (\ref{iter3b}) 
\begin{equation*}
\psi _{1}\left( f_{0}\left( x\right) \right) =\int_{0}^{G1\left( x\right)
}\psi (s)ds\leq G\left( \psi (\mathfrak{b}G1)\right) (x)=G\left( \phi
(G1)\right) (x)=f_{1}\left( x\right) .
\end{equation*}%
Let us make an inductive step from $k$ to $k+1$, where $k\geq 1$. Fix $y\in
\Omega $ and define the set 
\begin{equation*}
\Omega _{y}=\{z\in \Omega \colon f_{k}(z)\leq f_{k}(y)\}.
\end{equation*}%
Then by (\ref{fT}) we have 
\begin{equation*}
G\left( 1_{\Omega _{y}}\phi \left( f_{k-1}\right) \right) \left( z\right)
\leq f_{k}\left( z\right) \leq f_{k}\left( y\right) \ \ \text{for all }z\in
\Omega _{y},
\end{equation*}%
which implies by the weak maximum principle that 
\begin{equation}
G\left( 1_{\Omega _{y}}\phi \left( f_{k-1}\right) \right) (x)\leq \mathfrak{b%
}\,f_{k}(y)\ \ \text{for all }x\in \Omega \text{.}  \label{max-j}
\end{equation}%
Consider now the kernel 
\begin{equation*}
\hat{G}\left( x,dz\right) =G\left( x,1_{\Omega _{y}}dz\right),
\end{equation*}%
and define a sequence of functions  $\{\hat{f}_{k}\}$ similarly to (\ref{fT}):%
\begin{equation*}
\hat{f}_{0}=\hat{G}1=G1_{\Omega _{y}},\text{\ \ \ }\hat{f}_{k+1}=\hat{G}%
(\phi (\hat{f}_{k}))=G\left( 1_{\Omega _{y}}\phi (\hat{f}_{k})\right) , \, \, k 
\ge 0.
\end{equation*}%
It follows from (\ref{max-j}) that%
\begin{equation*}
\hat{f}_{k}\left( x\right) \leq \mathfrak{b}\,f_{k}(y)\ \ \text{for all }%
x\in \Omega .
\end{equation*}%
By the inductive hypothesis, we have, for all $x\in \Omega $,%
\begin{equation*}
\psi _{k}\left( \hat{f}_{0}\left( x\right) \right) \leq \hat{f}_{k}(x).
\end{equation*}%
It follows that%
\begin{equation*}
\psi \circ \psi _{k}\left( \hat{f}_{0}\left( x\right) \right) \leq \psi (%
\hat{f}_{k}(x))\leq \psi \left( \mathfrak{b}\,f_{k}(y)\right) =\phi \left(
f_{k}(y)\right) ,
\end{equation*}%
that is,
\begin{equation}
\psi \circ \psi _{k}\left( G1_{\Omega _{y}}\right) \left( x\right) \leq \phi
\left( f_{k}(y)\right) \ \text{for all }x\in \Omega .  \label{psik}
\end{equation}%
Fix now also $x\in \Omega $ and apply Lemma \ref{lemma1.2} with the $\sigma $%
-finite measure $\omega \left( dy\right) =G\left( x,dy\right) .$ We obtain,
using (\ref{psi-j}) and (\ref{psik}), that 
\begin{eqnarray*}
\psi _{k+1}\left( f_{0}\left( x\right) \right) &=&\int_{0}^{f_{0}\left(
x\right) }\psi \circ \psi _{k}\left( s\right) ds=\int_{0}^{\omega \left(
\Omega \right) }\psi \circ \psi _{k}\left( s\right) ds \\
&\leq &\int_{\Omega }\psi \circ \psi _{k}\left( \omega \left( z\in \Omega
:f_{k}\left( z\right) \leq f\left( y\right) \right) \right) \omega \left(
dy\right) \\
&=&\int_{\Omega }\psi \circ \psi _{k}\left( \omega \left( \Omega _{y}\right)
\right) \omega \left( dy\right) \\
&=&\int_{\Omega }\psi \circ \psi _{k}\left( G1_{\Omega _{y}}\left( x\right)
\right) \omega \left( dy\right) \\
&\leq &\int_{\Omega }\phi \left( f_{k}\left( y\right) \right) \omega \left(
dy\right) \\
&=&G\left( \phi \left( f_{k}\right) \right) \left( x\right) =f_{k+1}\left(
x\right) ,
\end{eqnarray*}%
which finishes the inductive step.
\end{proof}

Setting in Lemma \ref{lemma1.4} $\phi (t)=t^{q}$, $q>0$, we obtain the
following.

\begin{cor}
\label{cor1.4} Under the hypotheses of Lemma \emph{\ref{lemma1.4}}, we have,
for any $q>0$, $k \ge 0$, and all $x\in \Omega $, 
\begin{equation}
\left[ f_{0}(x)\right] ^{1+q+\cdots +q^{k}}\leq c(q,k)\,\mathfrak{b}%
^{q+q^{2}+\cdots +q^{k}}f_{k}(x),  \label{iter4}
\end{equation}%
where $f_{k}$ are defined by \emph{(\ref{fT})} with $\phi (t)=t^{q}$, and 
\begin{equation}
c(q,k)=\prod_{j=1}^{k}(1+q+\cdots +q^{j})^{q^{k-j}}.  \label{const4}
\end{equation}%
In particular, in the case $q=1$, for all $x\in \Omega $ we have 
\begin{equation}
\left[ f_{0}\left( x\right) \right] ^{k+1}\leq (k+1)!\,\mathfrak{b}%
^{k}\,f_{k}(x).  \label{iter4a}
\end{equation}
\end{cor}

\begin{rmk}
\label{const2}  {\rm A direct proof by induction using Lemma~\ref{lemma1.3a}
gives a constant that grows in $\mathfrak{b}$ much faster than $\mathfrak{b}%
^{q+q^{2}+\cdots +q^{k}}$ in (\ref{iter4}).}
\end{rmk}

\section{Monotone nonlinearities}

\label{nonlin-int}

In this section we will apply estimates of Sec.~\ref{weakmaximum} to the
following nonlinear problem. Let $\Omega $ and $G$ be as above, that is $%
(\Omega ,\Xi )$ is a measurable space and $G\left( x,dy\right) $ be a $%
\sigma $-finite kernel in $\Omega $, satisfying the weak maximum principle (%
\ref{wmp-f}). Let $g\colon \lbrack 1,+\infty )\rightarrow \lbrack 0,+\infty
) $ be a continuous, monotone non-decreasing function; set $g(+\infty
)=\lim_{t\rightarrow +\infty }g(t)$. We consider the non-linear integral
inequality 
\begin{equation}
u\left( x\right) \geq G\left( g\left( u\right) \right) \left( x\right) +1\ 
\text{for all }x\in \Omega ,\text{ }  \label{super-sol}
\end{equation}%
where $u:\Omega \rightarrow \lbrack 1,\infty )$ is a measurable function.
Our goal is to obtain sharp pointwise lower estimates of $u(x)$ which are
better than the trivial estimate $u(x)\geq 1$. In what follows we always
assume that $g\left( 1\right) \geq 1$.

\begin{rmk}
\label{rmk2.1} {\rm Indeed, if $g\left( 1\right) =0$, then simple examples,
for instance, $g(t)=\log t$ and $u\equiv 1$, show that we cannot expect any
non-trivial estimates for $u$. If $g(1)>0$, then by renaming $\frac{g}{%
g\left( 1\right) }$ back to $g$ and changing $G$ appropriately, we can
assume that $g\left( 1\right) =1$. Hence, the assumption $g\left( 1\right)
\geq 1$ is natural in this setting.}
\end{rmk}

The next theorem is our main result.

\begin{theorem}
\label{theorem2a} Let $G$ be a $\sigma $-finite kernel on $\Omega $
satisfying the weak maximum principle \emph{(\ref{wmp-f})} with $\mathfrak{b}%
\geq 1$. Let $g\colon \lbrack 1,+\infty )\rightarrow \lbrack 1,+\infty )$ be
a continuous monotone non-decreasing function. Set 
\begin{equation}
F(t)=\int_{1}^{t}\frac{ds}{g(s)},\quad t\geq 1.  \label{Psi-def}
\end{equation}%
If $u$ satisfies \emph{(\ref{super-sol})} then, for all $x\in \Omega $ such
that $u(x)<+\infty $, the following inequalities hold: 
\begin{equation}
u(x)\geq 1+\mathfrak{b}\left[ F^{-1}\left( \mathfrak{b}^{-1}G1(x)\right) -1%
\right]  \label{lower-psi}
\end{equation}%
and 
\begin{equation}
\mathfrak{b}^{-1}G1(x)<\,a:=\int_{1}^{+\infty }\frac{ds}{g(s)}.
\label{necessary-psi}
\end{equation}
\end{theorem}

Note that the function $F$ is defined on $[1,\infty )$. Hence, the inverse
function $F^{-1}$ is defined on $[0,a)$, and takes values in $[1,\infty )$.
Hence, the condition (\ref{necessary-psi}) is necessary for the right hand
side of (\ref{lower-psi}) to be well-defined.

\begin{proof}
Set for any $t\geq 0$ 
\begin{equation}
\phi \left( t\right) =g\left( t+1\right) \ \ \ \text{and\ \ }\psi (t)=\phi (%
\mathfrak{b}^{-1}t)=g(\mathfrak{b}^{-1}t+1).  \label{def-psi-g}
\end{equation}%
Define the sequence $\left\{ f_{k}\right\} $ of functions on $\Omega $ by (%
\ref{fT}), that is,%
\begin{equation*}
f_{0}=G1,\ \ \ f_{k+1}=G\left( \phi \left( f_{k}\right) \right) .
\end{equation*}%
We claim that, for all $k\geq 0$,%
\begin{equation}
u\geq f_{k}+1\ \ \text{in }\Omega .  \label{ufk}
\end{equation}%
Indeed, it follows from (\ref{super-sol}) that $u\geq 1$, and one more
application of (\ref{super-sol}) yields%
\begin{equation*}
u\geq G\left( g\left( 1\right) \right) +1\geq G1+1=f_{0}+1,
\end{equation*}%
that is, (\ref{ufk}) for $k=0$. If (\ref{ufk}) is already proved for some $%
k\geq 0$, then substituting (\ref{ufk}) into (\ref{super-sol}) yields%
\begin{equation*}
u\geq G\left( g\left( f_{k}+1\right) \right) +1=G\left( \phi \left(
f_{k}\right) \right) +1=f_{k+1}+1,
\end{equation*}%
which finishes the proof of (\ref{ufk}).

Consider now the sequence $\left\{ \psi _{k}\right\} _{k=0}^{\infty }$ of
functions on $[0,\infty )$ defined by (\ref{psi-j}), that is, $\psi
_{0}\left( t\right) =t$ and 
\begin{equation}
\psi _{k+1}(t)=\int_{0}^{t}\psi \circ \psi _{k}(s)ds.  \label{psi-j-a}
\end{equation}%
By Lemma \ref{lemma1.4}, we have, for all $x\in \Omega $ and $k\geq 0$, 
\begin{equation*}
f_{k}\left( x\right) \geq \psi _{k}\left( f_{0}\left( x\right) \right) ,
\end{equation*}%
which together with (\ref{ufk}) imply%
\begin{equation*}
u\left( x\right) \geq \psi _{k}\left( G1\left( x\right) \right) +1\ \ \text{%
for all }x\in \Omega .
\end{equation*}%
By (\ref{def-psi-g}) the function $\psi $ is monotone non-decreasing and $%
\psi \geq 1$, which implies that the sequence $\left\{ \psi _{k}\right\}
_{k=0}^{\infty }$ is non-decreasing, that is, $\psi _{k+1}\left( t\right)
\geq \psi _{k}\left( t\right) $ for all $t\geq 0$. Indeed, for $k=0$ it
follows from 
\begin{equation*}
\psi _{1}\left( t\right) =\int_{0}^{t}\psi \left( t\right) dt\geq t=\psi
_{0}\left( t\right) ,
\end{equation*}%
and if $\psi _{k}\geq \psi _{k-1}$ is already proved then $\psi _{k+1}\geq
\psi _{k}$ follows from (\ref{psi-j-a}) and the monotonicity of $\psi .$

Set%
\begin{equation*}
\psi _{\infty }\left( t\right) =\lim_{k\rightarrow \infty }\psi _{k}\left(
t\right)
\end{equation*}%
so that%
\begin{equation}
u\left( x\right) \geq \psi _{\infty }\left( G1\left( x\right) \right) +1\ \ 
\text{for all }x\in \Omega .  \label{est-j-infty}
\end{equation}%
Let us fix $x\in \Omega $ such that $u(x)<+\infty $. It follows from (\ref%
{est-j-infty}) that 
\begin{equation*}
t_{0}:=G1(x)<+\infty \ \ \text{and\ \ \ }\psi _{\infty }\left( t_{0}\right)
<\infty .
\end{equation*}
Without loss of generality we may assume that $t_{0}>0$ since in the case $%
G1(x)=0$ the estimates (\ref{lower-psi}), (\ref{necessary-psi}) are obvious.
Then we see that the function $\psi _{\infty }$ is finite on $\left[ 0,t_{0}%
\right] $, positive on $(0,t_{0}]$ and satisfies the integral equation 
\begin{equation}
\psi _{\infty }(t)=\int_{0}^{t}\psi \circ \psi _{\infty }(s)\,ds,\quad
0<t\leq t_{0}  \label{iter-inf}
\end{equation}%
It follows that $\psi _{\infty }$ is continuously differentiable on $\left[
0,t_{0}\right] $ and satisfies the differential equation 
\begin{equation}
\frac{d\psi _{\infty }}{dt}=\psi (\psi _{\infty }(t)),\quad \psi _{\infty
}(0)=0.  \label{psi-inf}
\end{equation}%
Setting 
\begin{equation*}
\Psi (\xi )=\int_{0}^{\xi }\frac{ds}{\psi (s)}=\mathfrak{b}\,F(1+\mathfrak{b}%
^{-1}\xi )
\end{equation*}%
and observing that by (\ref{psi-inf}) 
\begin{equation*}
\frac{d\Psi (\psi _{\infty })(t)}{dt}=1,
\end{equation*}%
we obtain, for any $t\in \left[ 0,t_{0}\right] $, 
\begin{equation}
\Psi \left( \psi _{\infty }(t)\right) =t.  \label{Psi}
\end{equation}%
It follows that, for $t=t_{0}$,%
\begin{equation}
F(1+\mathfrak{b}^{-1}\psi _{\infty }\left( t_{0}\right) )=\mathfrak{b}%
^{-1}t_{0}.  \label{Ft}
\end{equation}%
Since all values of $F$ are located in $[0,a)$, we obtain that%
\begin{equation*}
\mathfrak{b}^{-1}t_{0}<a,
\end{equation*}%
which is equivalent to (\ref{necessary-psi}). Next, we obtain from (\ref{Ft}%
) that%
\begin{equation*}
\psi _{\infty }\left( t_{0}\right) =\mathfrak{b}F^{-1}\left( \mathfrak{b}%
^{-1}t_{0}\right) -1.
\end{equation*}%
Substituting this into (\ref{est-j-infty}) yields (\ref{lower-psi}), which
finishes the proof.
\end{proof}

Our methods are also applicable to non-linear integral inequalities of the
type 
\begin{equation}
u(x)+G(g(u))\leq 1\ \,\text{in}\,\,\Omega ,  \label{sub-sol}
\end{equation}%
defined on measurable functions $0\leq u\leq 1$. Here $g\colon
(0,1]\rightarrow (0,+\infty )$ is a continuous, monotone non-increasing
function {\color{red} and} $%
g(0)=\lim_{t\rightarrow 0^{+}}g(t)\leq +\infty $.

We exclude the case $g(1)=0$, since otherwise we cannot expect an upper
estimate for $u$ better than the trivial estimate $u\le 1$. In fact, without
loss of generality we may assume $g(1)\ge 1$ by using $\frac{g}{g(1)}$ and
changing $G$ appropriately as above (see Remark~\ref{rmk2.1}). Thus, we
assume that $g\colon [0, 1] \to [1, +\infty].$

\begin{theorem}
\label{theorem3a} Let $G$ be a $\sigma$-finite kernel on $\Omega $
satisfying the weak maximum principle \emph{(\ref{wmp-f})} with $\mathfrak{b}%
\geq 1$. Let $g: \lbrack 0, 1 \rbrack \rightarrow \lbrack 1,+\infty]$ be a
continuous monotone non-increasing function. Set 
\begin{equation}  \label{Psi-def-sub}
F(t) = \int_t^1 \frac{ds}{g(s)}, \quad 0\le t\le 1.
\end{equation}
If $u\ge 0$ satisfies (\ref{sub-sol}) then, for all $x\in \Omega$ such that $%
u(x)>0$, the following inequalities hold: 
\begin{equation}  \label{lower-psi-sub}
u(x) \le 1- \mathfrak{b} \, \left[1-F^{-1} \Big(\mathfrak{b} ^{-1}G 1 (x)%
\Big)\right],
\end{equation}
and 
\begin{equation}  \label{necessary-psi-sub}
\mathfrak{b}^{-1} \, G 1 (x) < F(1- \mathfrak{b}^{-1})= \int_{1- \mathfrak{b}%
^{-1}}^{1} \frac{ds}{g(s)},
\end{equation}
so that the right-hand side of (\ref{lower-psi-sub}) is well-defined and
positive.
\end{theorem}

\begin{proof}
For any $t\in [0,1]$ set 
\begin{equation}  \label{def-psi-s}
\phi (t) = g(1-t) \quad \text{and} \quad \psi(t)= \phi (\mathfrak{b}%
^{-1}t)=g(1-\mathfrak{b}^{-1}t),
\end{equation}
so that $\phi$ and $\psi$ are non-decreasing. As in the proof of Theorem \ref%
{theorem2a}, define the sequence $\left\{ f_{k}\right\}_{k=0}^{\infty} $ of
functions on $\Omega $ by (\ref{fT}), 
\begin{equation*}
f_{0}=G1,\ \ \ f_{k+1}=G\left( \phi \left( f_{k}\right) \right) .
\end{equation*}%
We claim that $0\le f_k \le 1$ for all $k\geq 0$, so that $\phi \left(
f_{k}\right)$ is well-defined, and moreover 
\begin{equation}
u\leq 1- f_{k}\ \ \text{in }\Omega .  \label{ulk}
\end{equation}

Indeed, obviously $u\le 1$, and hence $g(u)\ge g(1)\ge 1$. Consequently by (%
\ref{sub-sol}), 
\begin{equation*}
u\leq 1- G\left( g\left( 1\right) \right) \leq 1- G1= 1-f_{0},
\end{equation*}%
so that $f_0\le 1$, and (\ref{ulk}) holds for $k=0$. If (\ref{ulk}) has been
proved for some $k\geq 0$, then substituting (\ref{ulk}) into (\ref{sub-sol}%
) yields%
\begin{equation*}
u\leq 1-G\left( g\left(1- f_{k}\right) \right) =1-G\left( \phi \left(
f_{k}\right) \right) =1-f_{k+1},
\end{equation*}%
which proves (\ref{ulk}). In particular, $0\le f_k\le 1$ for all $k\ge 0$.

Set $t_0:= G1(x)=f_0(x)\in \lbrack 0, 1\rbrack$. Consider now the sequence $%
\left\{ \psi _{k}\right\} _{k=0}^{\infty }$ of functions on $[0, t_0]$
defined by (\ref{psi-j}), 
\begin{equation*}
\psi _{0}\left( t\right) =t, \quad \psi _{k+1}(t)=\int_{0}^{t}\psi \circ
\psi _{k}(s)ds.  \label{psi-j-b}
\end{equation*}

We need to show that the functions $\psi_k(t)$ are well-defined, that is, $%
\psi_k(t) \in \lbrack 0, 1\rbrack$ for all $t \in \lbrack 0, t_0\rbrack$ and 
$k\ge 1$, and 
\begin{equation}
f_k(x) \ge \psi_k(f_0(x)).  \label{k-est}
\end{equation}

By (\ref{ulk}), we deduce 
\begin{equation}
G(g(u))\ge G(g(1-f_k))=G(\phi(f_k))=f_{k+1}.  \label{f-est}
\end{equation}
Clearly, $\psi_1$ is well-defined. Hence, by Lemma \ref{lemma1.4} and the
preceding inequality with $k=0$, 
\begin{equation*}
1\ge 1-u(x)\ge G(g(u))(x)\ge f_1(x)\ge \psi_1(f_0(x))=\psi_1(t_0).
\end{equation*}
Since $\psi_1$ is non-decreasing, we see that $\psi_1(t)\in \lbrack 0,
1\rbrack$ for all $t\in \lbrack 0, t_0\rbrack$.

If the inequalities $\psi_k(t_0)\le 1$ and (\ref{k-est}) have been proved
for some $k\ge 1$, then $\psi_{k+1}(t)$ is well-defined on $[0, t_0]$. By
Lemma \ref{lemma1.4} and (\ref{f-est}), 
\begin{equation*}
1\ge 1-u(x)\ge G(g(u))(x)\ge f_{k+1}(x)\ge \psi_{k+1}(f_0(x))=
\psi_{k+1}(t_0).
\end{equation*}
Since $\psi_{k+1}(t)$ is a non-decreasing function, it follows that $%
\psi_{k+1}(t)\le 1$ for all $t \in [0, t_0]$, and (\ref{k-est}) holds for
all $k\ge 1$. Consequently, for all $k \ge 0$, 
\begin{equation}
1\ge u(x) + f_k(x) \ge u(x) + \psi_k(f_0(x)).  \label{infty-est-s}
\end{equation}

The rest of the proof is similar to that of Theorem~\ref{theorem2a}. Passing
to the limit as $k\rightarrow \infty $ in (\ref{infty-est-s}) yields 
\begin{equation*}
1\geq u(x)+\psi _{\infty }(f_{0}(x)),
\end{equation*}%
where $\psi _{\infty }(t)$ is the unique solution of the integral equation (%
\ref{iter-inf}) in the interval $0\leq t\leq t_{0}$. From this we deduce
that 
\begin{equation*}
\Psi \Big(\psi _{\infty }(t)\Big)=t,\quad 0\leq t\leq t_{0},
\end{equation*}%
where, for $\xi\in\lbrack 0, 1 \rbrack$, 
\begin{equation*}
\Psi (\xi)=\int_{0}^{\xi}\frac{ds}{\psi (t)}=\mathfrak{b}\,F(1-\mathfrak{b}%
^{-1}\xi).
\end{equation*}%
Here $F$ is defined by (\ref{Psi-def-sub}). Hence, 
\begin{equation*}
\psi _{\infty }(t)=\Psi ^{-1}(t)=\mathfrak{b}\,\left( 1-F^{-1}(\mathfrak{b}%
^{-1}t)\right) ,\quad 0\leq t\leq t_{0},\quad t_{0}\in \mathrm{Range}(\Psi ).
\end{equation*}%
Thus, 
\begin{equation*}
G1(x)=t_{0}\leq \int_{0}^{1}\frac{dt}{\psi (t)}=\mathfrak{b}\int_{1-%
\mathfrak{b}^{-1}}^{1}\frac{dt}{g(t)},
\end{equation*}%
so that (\ref{lower-psi-sub}) holds, where its right-hand side is well
defined. Clearly, (\ref{lower-psi-sub}) implies (\ref{necessary-psi-sub})
provided $u(x)>0$.
\end{proof}

We now consider some special cases of the integral inequalities (\ref%
{super-sol}) and (\ref{sub-sol}). Let $g(t)=t^q$ $(q>0)$ in the case of (\ref%
{super-sol}), that is, 
\begin{equation}  \label{super-q}
u \ge G( u^q) +1 \quad \text{in}\, \, \Omega.
\end{equation}

\begin{cor}
\label{q-cor} Let $q>0$. Under the assumptions of Theorem~\ref{theorem2a},
suppose that $u$ satisfies \emph{(\ref{super-q})}.

If $q\not=1$, then 
\begin{equation}
u(x)\geq 1+\mathfrak{b}\Big[\Big(1+(1-q)\mathfrak{b}^{-1}G 1 (x)\Big)^{\frac{%
1}{1-q}}-1\Big],  \label{sharp-est}
\end{equation}%
where in the case $q>1$, for all $x\in \Omega $ such that $u(x)<+\infty $, 
\begin{equation}
\mathfrak{b}^{-1} G 1 (x)<\frac{1}{q-1}.  \label{necessary}
\end{equation}

In the case $q=1$, 
\begin{equation}
u(x)\geq 1+\mathfrak{b}\,\Big(e^{\mathfrak{b}^{-1}\,G 1 (x)}-1\Big), \quad
x\in \Omega .  \label{sharp-est-1}
\end{equation}
\end{cor}

\begin{proof}
We apply Theorem~\ref{theorem2a} with $g(t)=t^{q}$, $t\geq 1$. If $q>0$ ($%
q\not=1$), we have 
\begin{equation*}
F(t)=\int_{1}^{t}s^{-q}ds=\frac{1}{1-q}\Big(t^{1-q}-1\Big),\quad t\geq 1,
\end{equation*}%
and consequently 
\begin{equation*}
F^{-1}(\tau )=[1+(1-q)\,\tau ]^{\frac{1}{1-q}},\,\,\,0\leq \tau \leq \frac{1%
}{q-1}\,\,\,\text{if}\,\,\,q>1;\,\,\,\tau \geq 0\,\,\,\text{if}\,\,\,0<q<1.
\end{equation*}%
Therefore, for all $x\in \Omega $ such that $u(x)<+\infty $, we deduce (\ref%
{sharp-est}), where $G1 (x)<\frac{\mathfrak{b}}{1-q}$, so that the
right-hand side of (\ref{sharp-est}) is well defined, that is, (\ref%
{necessary}) holds.

In the case $q=1$, we have $F(t)=\log t$ for $t\geq 1$, and $F^{-1}(\tau
)=e^{\tau }$ for $\tau \geq 0$, which gives (\ref{sharp-est-1}).
\end{proof}

We now consider inequalities of the type (\ref{sub-sol}) with $g(t)=t^{q}$
for $q<0$, that is, 
\begin{equation}  \label{sub-q}
{\color{red} u + G( u^{q}) \le 1} \quad \text{in} \, \, \Omega.
\end{equation}

\begin{cor}
\label{q-cor-s} Let $q<0$. Under the assumptions of Theorem~\ref{theorem3a},
suppose that $u\geq 0$ satisfies \emph{(\ref{sub-q})}. Then, for all $x\in
\Omega $ such that $u(x)>0$, 
\begin{equation}
G 1 (x)<\frac{\mathfrak{b}}{1-q}\Big[1-(1-\mathfrak{b}^{-1})^{1-q}\Big],
\label{necessary-s}
\end{equation}
and 
\begin{equation}
u(x)\leq 1-\mathfrak{b}\Big[1-\Big(1-(1-q)\mathfrak{b}^{-1}G 1 (x)\Big)^{%
\frac{1}{1-q}}\Big].  \label{sharp-est-s}
\end{equation}
\end{cor}

\begin{proof}
We apply Theorem~\ref{theorem3a} with $g(t)=t^{q}$, $q<0$, where $t\in (0,1]$%
. In this case, 
\begin{equation*}
F(t)=\int_{t}^{1}s^{-q}ds=\frac{1}{1-q}\Big(1-t^{1-q}\Big),\quad 0\leq t\leq
1.
\end{equation*}%
Then 
\begin{equation*}
F^{-1}(\tau )=[1-(1-q)\,\tau ]^{\frac{1}{1-q}},\quad 0\leq \tau \leq \frac{1%
}{1-q}.
\end{equation*}%
Therefore, for all $x\in \Omega $ such that $u(x)>0$, we deduce (\ref%
{sharp-est-s}), where $G 1 (x)\leq \frac{\mathfrak{b}}{1-q}$ so that the
right-hand side of (\ref{sharp-est-s}) is well defined. Moreover, (\ref%
{necessary-s}) holds since $u(x)>0$ in (\ref{sharp-est-s}).
\end{proof}

In the following corollary we give pointwise estimates for super-solutions
to \textit{homogeneous} equations in the sublinear case.

\begin{cor}
\label{thm2.1} Let $G$ be a $\sigma$-finite kernel on $\Omega $ satisfying
the weak maximum principle \emph{(\ref{wmp-f})} with $\mathfrak{b}\geq 1$.
Let $0<q<1$. If $u>0$ in $\Omega$ satisfies 
\begin{equation}  \label{supersol2}
u \ge G(u^q) \quad \emph{\text{in}} \, \, \Omega,
\end{equation}
then 
\begin{equation}  \label{lowerest2}
u (x) \ge (1-q)^{\frac{1}{1-q}} \mathfrak{b}^{-\frac{q}{1-q}} \left( G 1
(x)\right)^{\frac{1}{1-q}}, \quad x \in \Omega.
\end{equation}
\end{cor}

\begin{rmk}
\label{const=sharp} {\rm The constant $(1-q)^{\frac{1}{1-q}}$ in (\ref%
{lowerest2}) coincides with that in \cite{GV},Theorem 3.3, if $\mathfrak{b}%
=1 $, and is sharp.}
\end{rmk}

\begin{proof}
For $a>0$, set 
\begin{equation*}
E_a= \{y \in \Omega\colon \,\, u(y)\ge a\}
\end{equation*}
Then 
\begin{equation*}
u\ge G(u^q)\ge a^q \left(G (1_{E_a}u^q)\right) \quad \text{in} \,\, E_a.
\end{equation*}
Iterating this inequality, we obtain 
\begin{equation*}
u \ge a^{q^{k+1}} G f_k \quad \text{in} \,\, E_a,
\end{equation*}
where $f_k$ is defined by (\ref{fT}) with $\phi(t)=t^q$. Hence, by Corollary %
\ref{cor1.4}, 
\begin{equation*}
u\ge c(q, k)^{-1} \, a^{q^{k+1}} \mathfrak{b}^{-q(1+q + \cdots + q^{k-1})}
\left(G 1_{E_a}(x)\right)^{1+q + \cdots + q^k} \quad \text{in} \,\, E_a.
\end{equation*}

Notice that 
\begin{align*}
c(q, k) & = \prod_{j=1}^{k} (1+q + \cdots + q^j)^{q^{k-j}} \\
& < \prod_{j=1}^{k} (1-q)^{-q^{k-j}} \\
& < (1-q)^{-(1-q)^{-1}}.
\end{align*}

It follows 
\begin{equation*}
u\ge (1-q)^{(1-q)^{-1}} a^{q^{k+1}} \mathfrak{b}^{-q(1+q + \cdots +
q^{k-1})} \left(G 1_{E_a}(x)\right)^{1+q + \cdots + q^k} \quad \text{in}
\,\, E_a.
\end{equation*}
Letting $k \to +\infty$, we obtain 
\begin{equation*}
u\ge (1-q)^{\frac{1}{1-q}} \mathfrak{b}^{-\frac{q}{1-q}} (G 1_{E_a})^{\frac{1%
}{1-q}} \quad \text{in} \,\, E_a.
\end{equation*}
Finally, letting $a \to 0^{+}$ yields (\ref{lowerest2}) in $\Omega$.
\end{proof}

\section{Quasi-metric kernels}

\label{quasimetric}

Consider the setting of Example \ref{Ex2}, that is, $\Omega $ is a locally
compact Hausdorff space with countable base, $\omega $ is a Radon measure on 
$\Omega $. Let $K:\Omega \times \Omega \rightarrow (0,+\infty ]$ be a
lower semi-continuous function. Assume that for any $x\in
\Omega $, $K\left( x,\cdot \right) \in L_{loc}^{1}\left( \Omega ,\omega
\right) $. Set 
\begin{equation*}
Gf(x)=\int_{\Omega }K(x,y)\,f(y)d\omega (y).
\end{equation*}

The kernel $K$ is called \textit{quasi-metric }(see \cite{KV}, \cite{H}, 
\cite{FNV}) if $K$ is symmetric, that is, $K\left( x,y\right) =K\left(
y,x\right) $ and $d\left( x,y\right) =\frac{1}{K\left( x,y\right) }$ is
quasi-metric, that is, there exists a \textit{quasi-metric constant} $%
\varkappa >0$ such that the quasi-triangle inequality holds: 
\begin{equation*}
d(x,y)\leq \varkappa \,[d(x,z)+d(y,z)],\quad \forall \,x,y,z\in \Omega .
\end{equation*}%
Without loss of generality we may assume that $d(x, y)\not=0$ for some $x, y
\in \Omega$, so that $\varkappa \geq \frac{1}{2}$.

The following lemma is proved in \cite{QV2} (Lemma 3.5).

\begin{lemma}
\label{lemma-wmp-q}Suppose $G$ is a quasi-metric kernel in $\Omega $ with
quasi-metric constant $\varkappa $. Then $G$ satisfies the weak maximum
principle \emph{(\ref{wmp})} with constant $\mathfrak{b}=2\varkappa $.
\end{lemma}

In the next lemma we consider a certain modification of a quasi-metric
kernel.

\begin{lemma}
\label{lemma-elem} Suppose $K$ is a quasi-metric kernel in $\Omega$ with
constant $\varkappa$. For $w \in \Omega$, let $\Omega_w= \{x \in
\Omega\colon \, \, K(x, w)<+\infty\}$. Then 
\begin{equation}  \label{eq-w}
K_w(x, y) = \frac{K(x, y)}{K(x, w) \, K(y, w)}, \quad x, y \in \Omega_w,
\end{equation}
is a quasi-metric kernel on $\Omega_w$ with quasi-metric constant $%
4\varkappa^2$.

In particular, $K_w$ satisfies the weak maximum principle \emph{(\ref{wmp})} in $\Omega_w$ with
constant $\mathfrak{b} = 8 \varkappa^3$.
\end{lemma}

\begin{proof}
This is immediate from the so-called Ptolemy inequality for quasi-metric
spaces, \cite{FNV}, Lemma 2.2 (see also \cite{HN}, Proposition 8.1), 
\begin{equation}
d(x,y)\,d(z,w)\leq 4\varkappa
^{2}[d(x,w)\,d(y,z)+d(x,z)\,d(y,w)],\,\,\forall \,x,y,z,w\in \Omega .
\label{ptolemy}
\end{equation}%
Dividing both sides of the preceding inequality by $d(x,w)\,d(y,w)\,d(z,w)$,
we deduce 
\begin{equation*}
\frac{d(x,y)}{d(x,w)\,d(y,w)}\leq 4\varkappa ^{2}\left[ \frac{d(x,z)}{%
d(x,w)\,d(z,w)}+\frac{d(y,z)}{d(y,w)\,d(z,w)}\right] ,
\end{equation*}%
for all $x,y,z\in \Omega _{w}$.
\end{proof}

Let $h:\Omega \rightarrow (0,+\infty )$ be a l.s.c. function on $\Omega $.
For a general kernel $K:\Omega \times \Omega \rightarrow (0,+\infty ]$,
consider a modified kernel 
\begin{equation}
K^{h}(x,y)=\frac{K(x,y)}{h(x)\,h(y)}\quad \text{for}\,\,x,y\in \Omega .
\label{mod}
\end{equation}%
Here we discuss the question how to verify the weak maximum principle for $%
K^{h}$.

\begin{rmk}
\label{rmk1} {\rm As we will demonstrate below (see Lemma~\ref{lemma-wmp}), $%
K^{h}$ satisfies the weak maximum principle (\ref{wmp}) provided $K$ satisfies the
following form of the \textit{weak domination principle}:

\emph{Given a positive  l.s.c. function $h$ in $\Omega $, 
\begin{equation}
K\mu (x)\leq M\,h(x)\quad \forall \,x\in \mathrm{supp}(\mu
)\,\,\Longrightarrow \,\,K\mu (x)  
\leq \mathfrak{b}\,M\,h(x)\quad \forall
\,x\in \Omega ,  \label{dom}
\end{equation}%
for any compactly supported Radon measure $\mu$ with finite energy in $\Omega$, i.e., 
$\int_\Omega K \mu \, d \mu<+\infty$, and
any constant $M>0$.}

This property is sometimes called a \textit{dilated}
domination principle (see, e.g., \cite{K1}). In the case where (\ref{dom}) holds 
with $\mathfrak{b}=1$ for any 
$h=K\nu +a$, where $\nu$ is a Radon measure and $a\geq 0$ is a constant,  
it  is called the \textit{complete} maximum principle (see e.g., \cite{BH}, \cite{H}).

The weak domination  principle holds for Green's kernels associated with a large class of local and non-local operators, and super-harmonic $h$.}
 \end{rmk}

\begin{rmk}
\label{rmk2} {\rm It is easy to see that, for a quasi-metric kernel $K$, the
modified kernel $K^h$ with  $h = K \nu>0$, where $\nu$ is a Radon measure,
is generally not quasi-metric. However, it does satisfy the weak maximum principle
(\ref{wmp}) under some mild assumptions. See Lemma~\ref{lemma-qm} below.

The modified kernel $K^h$ in this case is essentially quasi-metric if $\nu$ is a measure supported at a single point $w \in \Omega$,
i.e., when $h(x)= c \, K(x, w)$, $c>0$, by Lemma \ref{lemma-elem}.}
\end{rmk}

Let us denote by $M^{+}(\Omega )$ the class of Radon measures in $\Omega$.

\begin{lemma}
\label{lemma-wmp} Suppose $K$ is a non-negative l.s.c. kernel in $\Omega $ which satisfies the
domination principle \emph{(\ref{dom})}. Suppose $h$ is a positive 
l.s.c. function in $\Omega$. Then the modified kernel $K^h$ defined by 
\emph{(\ref{mod})} satisfies the weak maximum principle \emph{(\ref{wmp})} in 
$\Omega^{\prime}=\Omega\setminus\{x: \, h(x)<+\infty\}$ with the
same constant $\mathfrak{b}$.

In particular, if \emph{(\ref{dom})} holds for $K$ and $h$ with $\mathfrak{b}=1$, then $K^h$ satisfies the strong maximum principle in $\Omega ^{\prime }$.
\end{lemma}

\begin{proof}
For $\nu \in M^{+}(\Omega^{\prime })$ with compact support,  let $d\tilde{\nu}=\frac{1}{h(x)} d\nu $. Notice that $h$ is bounded below by a positive constant on any compact
set so that $\tilde{\nu}\in M^{+}(\Omega )$, and 
$\mathrm{supp}(\tilde{\nu})\subseteq \mathrm{supp}(\nu)$.

Let 
\begin{equation}
\Omega_{m}\mathrel{\mathop:}=\{x\in \Omega \colon h(x)\leq m\},\quad
m=1,2,\ldots  \label{h-def-m}
\end{equation}%
Clearly, each $\Omega _{m}$ is a closed set, and $\Omega^{\prime
}=\bigcup_{m=1}^{+\infty }\Omega _{m}$. Let 
$d\nu_{m}=\chi_{\Omega_{m}}d \nu $,
 so that $\mathrm{supp}(\nu_{m})\subseteq \Omega_{m}$.

Suppose that, for a positive constant $M$, 
\begin{equation*}
K^h  \nu (x)\leq M \,\,\,\text{for}\,\,\text{all}\,\,x\in \text{$\mathrm{supp}$}%
(\nu ).
\end{equation*}%
Then obviously 
\begin{equation*}
K^h\nu _{m}(x)\leq M \, \,\,\text{for}\,\,\text{all}\,\,x\in \text{$\mathrm{supp}$}%
(\nu _{m}).
\end{equation*}%
It follows that 
\begin{equation*}
K \tilde\nu_{m}(x)\leq M \, h(x) \,\,\text{for}\,\,\text{all}\,\,x\in \text{$%
\mathrm{supp}$}(\tilde{\nu}_{m}),\quad \text{where}\quad d\tilde{\nu}_{m}(x)=%
\frac{1}{h\left( x\right) }d\nu _{m}.
\end{equation*} 
Clearly, $\tilde{\nu}_{m}$ has finite energy with respect to $K$, since 
\[
\int_{\Omega} K \tilde\nu_{m}\, d \tilde\nu_{m}\le M  \int_{\Omega}h d \tilde\nu_m = M  \int_{\Omega}d \nu_m<+\infty .
\] Hence, by (\ref{dom}) 
\begin{equation*}
K\tilde{\nu}_{m}(x)\leq \mathfrak{b}\,M\,h(x)\quad \forall x\in \Omega .
\end{equation*}%
Consequently, by the monotone convergence theorem 
\begin{equation*}
K\tilde{\nu}(x)\leq \mathfrak{b}\,M\,h(x)\quad \forall x\in \Omega ^{\prime}.
\end{equation*}%
Equivalently, 
\begin{equation*}
K^h \nu (x)\leq \mathfrak{b}\,M\quad \forall x\in \Omega ^{\prime}.
\end{equation*}%
Thus, $K^h$ satisfies the weak maximum principle (respectively, the strong
maximum principle if $\mathfrak{b}=1$) in $\Omega ^{\prime}$.
\end{proof}

\begin{lemma}
\label{lemma-qm} Let $K$ be a quasi-metric kernel on $\Omega\times\Omega$,
continuous in the extended sense. Let $h = K \nu $ where $\nu \in
M^{+}(\Omega)$, $h \not\equiv +\infty$. Then the modified kernel $K^h$ defined
by \emph{(\ref{mod})} satisfies the weak maximum principle in $\Omega^{\prime}=\{x
\in \Omega: \, \, h(x)<+\infty\}$.
\end{lemma}

\begin{proof}
If $\nu=\delta_{w}$ for some $w\in \Omega$ then by Lemma~\ref{lemma-elem}
the modified kernel $K_w$ given by (\ref{eq-w}) is a quasi-metric kernel on $%
\Omega_w= \{x \in \Omega\colon \, \, K(x, w)<+\infty\}$. By Lemma~\ref%
{lemma-wmp-q}, $K_w$ satisfies the weak maximum principle with constant $%
\mathfrak{b}=8 \kappa^3$ in $\Omega^{\prime}=\Omega_w$.

To show that, for general $h = K \nu$, the modified kernel $K^h$ defined by (%
\ref{mod}) satisfies the weak maximum principle in $\Omega^{\prime}$, we
invoke the idea used in \cite{N}, Theorem 7 (see also \cite{K1}, \cite{K2})
which reduces it to the \textit{elementary domination principle} in the case 
$\nu=\delta_w$.

Suppose first that $\mu \in M^{+}(\Omega )$ is a measure with compact
support and of finite energy, and $h=K\nu $. Let us show that $K$ satisfies (%
\ref{dom}) for $\mu \in M^{+}(\Omega )$. To this end, we argue by
contradiction. Assume that 
\begin{equation}
K\mu \leq K\nu \quad \text{on}\,\,F=\text{$\mathrm{supp}$}\,(\mu ),
\label{contr-a}
\end{equation}%
but there exists $w\in \Omega \setminus F$ such that 
\begin{equation}
K\mu (w)>\mathfrak{b}\,K\nu (w),  \label{contr-b}
\end{equation}%
where without loss of generality we may let $K\nu (w)<+\infty $.

Notice that quasi-metric kernels are symmetric, and strictly positive.
Hence, $\text{cap}(F)<+\infty $ for any compact set $F\subset \Omega $ (see 
\cite{F1}), and there exists an extremal measure $\mu _{F}\in M^{+}(\Omega )$
of finite energy, with $\mathrm{supp}\,(\mu _{F})\subseteq F$, such that by 
\cite{N}, Lemma~$1^{\ast }$ (see also \cite{K1}, \cite{K2}), 
\begin{equation}
K\mu _{F}(z)\leq K(z,w),\quad \forall z\in \text{$\mathrm{supp}$}\,(\mu
_{F}),  \label{naim-a}
\end{equation}%
and 
\begin{equation}
K\mu _{F}(z)\geq K(z,w)\ \ \ \text{on}\,\,F.  \label{naim-b}
\end{equation}

Since $K_w$ is a quasi-metric kernel in $\Omega_w$, it satisfies the weak
maximum principle with constant $\mathfrak{b}=8 \varkappa^3$, and
consequently $G$ satisfies the domination principle (\ref{dom}) with $\nu=
\delta_w$ and the same constant $\mathfrak{b}$ in $\Omega_w$. In fact, the
domination principle for $\mu_F$ and $\nu= \delta_w$ holds in $\Omega$,
i.e., 
\begin{equation}  \label{naim-c}
K \mu_F (x) \le \mathfrak{b} \, K(x, w), \quad \forall x \in \Omega,
\end{equation}
where $\mathfrak{b}=8 \kappa^3$, since the right-hand side of (\ref{naim-c})
is infinite on $\Omega\setminus \Omega_w$, and for all measures of finite
energy $\mu (\Omega\setminus \Omega_w)=0$. Indeed, by the quasi-triangle
inequality $K(x, y) = \frac{1}{d(x, y)}=+\infty$ if $x, y\in \Omega\setminus
\Omega_w$, and so $K \mu=+\infty$ on $\Omega\setminus \Omega_w$, unless $%
\mu(\Omega\setminus \Omega_w)=0$.

We denote by $\mathcal{E} (\mu, \nu)$ the mutual energy of the measures $%
\mu, \nu \in M^{+}(\Omega)$: 
\begin{equation}  \label{mutual}
\mathcal{E} (\mu, \nu) \mathrel{\mathop:}= \int_{\Omega} K\nu \, d \mu =
\int_{\Omega} K\mu \, d \nu.
\end{equation}

Let us estimate the mutual energy $\mathcal{E} (\mu_F, \nu)$. Integrating
both sides of (\ref{naim-c}) against $d \nu$ we deduce 
\begin{align*}
\mathcal{E} (\mu_F, \nu)& =\int_{\Omega} K\mu_F \, d \nu \\
& \le \mathfrak{b} \int_{\Omega} K(x, w) \, d \nu(x) = \mathfrak{b} \, G
\nu(w).
\end{align*}

On the other hand, it follows from (\ref{naim-b}) and (\ref{contr-b}) that 
\begin{align*}
\mathcal{E} (\mu_F, \mu)& =\int_{F} K \mu_F \, d \mu \\
&\ge \int_{F} K (x, w) \, d \mu(x) \\
& = K \mu(w) > \mathfrak{b} \, K \nu (w).
\end{align*}

Since $\mathcal{E} (\mu_F, \nu) \ge \mathcal{E} (\mu_F, \mu)$ by (\ref%
{contr-a}), we arrive at a contradiction.

Suppose now that $\mu \in M^{+}(\Omega )$ has compact support $F\subset
\Omega^{\prime }$, and $h=K \nu $. Then for 
$\Omega _{m}\subset \Omega^{\prime }$ defined by (\ref{h-def-m}) and 
$d\mu _{m}=\chi _{\Omega_{m}}d\mu $ we have 
\begin{equation*}
K\mu _{m}\leq h\leq m\quad \text{in}\,\,\Omega _{m}.
\end{equation*}%
Consequently, $\mu _{m}\in M^{+}(\Omega )$ has finite energy, $\text{\textrm{%
supp}}(\mu _{m})\subset F\cap \Omega _{m}$ is a compact set, and by the
previous case 
\begin{equation*}
K\mu_{m}(x)\leq \mathfrak{b}\,h(x),\quad x\in \Omega ^{\prime },
\end{equation*}%
for $m$ large enough. Passing to the limit as $m\rightarrow +\infty $ we
obtain by the monotone convergence theorem 
\begin{equation*}
K\mu (x)\leq \mathfrak{b}\,h(x),\quad x\in \Omega ^{\prime }.
\end{equation*}
\end{proof}

\section{The weak domination principle and nonlinear integral inequalities}

\label{supersol}

In the setting of Example \ref{Ex1}, let $\Omega $ be a locally compact
Hausdorff space with countable base, and let $G\left( x,dy\right) $ be a
Radon kernel in $\Omega $. Let $h:\Omega \rightarrow (0,+\infty )$ be a
given positive lower semi-continuous function in $\Omega $. In particular, $%
\inf_{F}h>0$ for every compact set $F\subset \Omega $.

In this section we consider super-solutions $u:\Omega \rightarrow \lbrack
0,+\infty )$ of 
\begin{equation}
u(x)\geq G(u^{q})(x)+h(x)\text{\ in }\Omega ,  \label{superlin-ineq}
\end{equation}%
in the case $q>0$, and sub-solutions $u:\Omega \rightarrow (0,+\infty )$ of 
\begin{equation}
u(x)\leq -G(u^{q})(x)+h(x)\ \ \text{in}\,\,\Omega ,  \label{sublin-ineq}
\end{equation}%
in the case $q<0$.

We will assume that $G$ satisfies the \textit{weak domination principle} in
the following form:

\emph{For any bounded measurable function $f$ with compact support,} 
\begin{equation}
G f(x)\leq h(x)\ \text{in }\textrm{supp} (f)\ \ \Longrightarrow \ \ Gf(x) \leq 
\mathfrak{b}\ h(x) \, \, \text{in\ }\Omega .  \label{wdp-a}
\end{equation}

Our main result in this setup is as follows.

\begin{theorem}
\label{thm3.1} In the above setting, for a given function $h$, assume that $G
$ satisfies the weak domination principle \emph{(\ref{wdp-a})} in $\Omega $.
Suppose that $u\geq 0$ satisfies \emph{(\ref{superlin-ineq})} if $q>0$, or $%
u>0$ and satisfies \emph{(\ref{sublin-ineq})} if $q<0$. Then $u\left(
x\right) $ satisfies the following estimates for all $x\in \Omega $:

$\left( i\right) $ If $q>0$, $q\not=1$, then 
\begin{equation}
u(x)\geq h(x)\left\{ 1+\mathfrak{b}\left[ \left( 1+\frac{(1-q)\,G(h^{q})(x)}{%
\mathfrak{b}\,h(x)}\right) ^{\frac{1}{1-q}}-1\right] \right\} ,
\label{sharp-esta}
\end{equation}%
where in the case $q>1$ necessarily 
\begin{equation}
G(h^{q})(x)<\frac{\mathfrak{b}}{q-1}\,h(x).  
\label{necessary-a}
\end{equation}

$\left( ii\right) $ If $q=1$, then 
\begin{equation}
u(x)\geq h(x)\Big[1+\mathfrak{b}\,\Big(e^{\mathfrak{b}^{-1}\,\frac{G(h)(x)}{%
h(x)}}-1\Big)\Big].  \label{sharp-est-1a}
\end{equation}

$\left( iii\right) $ If $q<0$, then 
\begin{equation}
u(x)\leq h(x)\left\{ 1-\mathfrak{b}\Big[1-\Big(1-\frac{(1-q)\,G(h^{q})(x)}{%
\mathfrak{b}\,h(x)}\Big)^{\frac{1}{1-q}}\Big]\right\}  \label{sharp-est-1s}
\end{equation}%
and necessarily 
\begin{equation}
G(h^{q})(x)<\frac{\mathfrak{b}}{1-q}\Big[1-(1-\mathfrak{b}^{-1})^{1-q}\Big]%
h(x).  \label{necessary-s-h}
\end{equation}
\end{theorem}

\begin{proof}
Suppose first that $q>0$. Let us consider a modified kernel 
\begin{equation*}
G^{h}\left( x, dy\right) =\frac{h\left( y\right) ^{q}}{h\left( x\right) }%
G\left( x,dy\right) .
\end{equation*}%
Clearly, $G^{h}$ is also a Radon kernel on any subset 
\begin{equation}
\Omega_{m}=\left\{ x\in \Omega: \, \, h\left( x\right) \leq m\right\}, \quad
m\ge 1 .  \label{Omm}
\end{equation}
Notice that each $\Omega_{m}$ is closed, $\Omega_m \subseteq \Omega_{m+1}$,
and $\displaystyle{\bigcup_{m=1}^{\infty}} \Omega_{m}=\Omega $.

Setting $G^{h}_m=G^{h}(x, 1_{\Omega_{m}}dy)$ and 
\begin{equation}
v(x)\mathrel{\mathop:}=\frac{u(x)}{h(x)},\quad x\in \Omega,  \label{v-def}
\end{equation}%
we see that $v$ satisfies the inequality 
\begin{equation}
v(x)\geq G^{h}_m (v^{q})(x)+1\ \ \text{for all }x\in \,\Omega.
\label{mod-ineq}
\end{equation}%
Moreover, $G^{h}_m$ satisfies the weak maximum principle (\ref{wmp-c}) in $%
\Omega$ with the same constant $\mathfrak{b}$, that is, for any bounded
measurable function $f$ with compact support in $\Omega$,  
\begin{equation}
\frac{1}{h} \, G (1_{\Omega_m} h^q f) \leq 1\ \text{in } \, \, \textrm{supp} (f)\ \
\Longrightarrow \ \ \frac{1}{h} \, G (1_{\Omega_m} h^q f) \leq \mathfrak{b}\
\, \, \text{in\ }\Omega .  \label{wmp-c-b}
\end{equation}
Indeed, this follows from the weak domination principle (\ref{wdp-a})
applied to $1_{\Omega_m} h^q f$ in place of $f$, which yields 
\begin{equation}
G (1_{\Omega_m} h^q f) \leq h\ \text{in }\textrm{supp} (f)\cap\Omega_m \ \
\Longrightarrow \ \ G (1_{\Omega_m} h^q f) \leq \mathfrak{b}\ h \, \, \text{%
in\ }\Omega.  \label{wmp-c-a}
\end{equation}
This proves (\ref{wmp-c-b}).

Hence, by Corollary~\ref{q-cor} with $G^{h}_m$ in place of $G$, it follows
from (\ref{mod-ineq}) 
that $v$ satisfies the following estimates for all $x \in \Omega$: 
\begin{equation}
v(x)\geq 1+\mathfrak{b}\left[ \left( 1+ \mathfrak{b}^{-1} (1-q)\,G^h_m
(1)(x)\right) ^{\frac{1}{1-q}}-1\right] ,  \label{sharp-estb}
\end{equation}
if $q>0$, $q\not=1$, where in the case $q>1$, we have 
\begin{equation}
G^h_m (1)(x)<\frac{\mathfrak{b}}{q-1} .  \label{necessary-b}
\end{equation}
If $q=1$, then 
\begin{equation}
v(x)\geq 1+\mathfrak{b}\,\left(e^{\mathfrak{b}^{-1} \, G^h_m(1)(x)}-1\right).  \label{sharp-est-1b}
\end{equation}

Passing to the limit as $m \to \infty$ we deduce by the monotone convergence
theorem that, for $q>0$, $q\not=1$,  
\[
v(x)\geq 1+\mathfrak{b}\left[ \left( 1+ \mathfrak{b}^{-1} (1-q)\,G^h
(1)(x)\right) ^{\frac{1}{1-q}}-1\right] ,  
\]
where the strict inequality holds for $q>1$ in 
\[
G^h (1)(x)<\frac{\mathfrak{b}}{q-1},
\]
  since in the preceding estimate 
 we have $v(x) =u(x) \, h(x)<+\infty$. If $q=1$, then 
\[
v(x)\geq 1+\mathfrak{b}\,\left(e^{\mathfrak{b}^{-1} \, G^h(1)(x)}-1\right).  
\]
 Going back from $v, G^h$ to $u, G$ in these estimates 
yields that (\ref{sharp-esta}) or (\ref{sharp-est-1a}) hold  at
every $x\in \Omega$, and in the case $q>1$ the necessary condition (\ref%
{necessary-a}) holds.

In the case $q<0$, estimates (\ref{sharp-est-1s}) and (\ref{necessary-s-h})
are deduced in a similar way from Corollary~\ref{q-cor-s}, provided $u(x)>0$.
\end{proof}

\begin{rmk}
\label{rmk3} {\rm The results  of Sec. \ref{quasimetric} show that, in that setup,  
the estimates of 
Theorem \ref{thm3.1} hold for quasi-metric kernels $K$ and $h=K\nu$ in 
$\Omega^{\prime}=\{x \in \Omega: \, h(x)<+\infty\}$,  
for all Radon measures $\nu$ in $\Omega$ such that $K\nu\not\equiv +\infty$.} 
\end{rmk}

\end{document}